\documentclass[12pt]{amsart}

\usepackage{amssymb,amsthm,amsmath,enumerate}
\usepackage[numbers,sort&compress]{natbib}
\usepackage{color}
\usepackage{graphicx}
\usepackage{amssymb,amsthm,amsmath}
\usepackage[numbers,sort&compress]{natbib}
\usepackage{color}
\usepackage{graphicx}

\newtheorem{thm}{Theorem}[section]
\newtheorem{cor}[thm]{Corollary}

\newcommand{\R}{\mathbb{R}}
\newcommand{\Z}{\mathbb{Z}}

\title[Quantitative Singularity]{Lower bounds on concentration through Borel transforms and quantitative singularity of   spectral measures near the arithmetic transition}

\author{Svetlana Jitomirskaya}
\address[Svetlana Jitomirskaya]{ Department of Mathematics, University of California, Berkeley, California 94720, USA}
\email{sjitomi@berkeley.edu}

\author{Wencai Liu}
\address[Wencai Liu]{Department of Mathematics, Texas A\&M University, College Station, TX 77843-3368, USA}\email{liuwencai1226@gmail.com; wencail@tamu.edu}

\author{Serguei  Tcheremchantsev}
\address[Serguei  Tcheremchantsev]{UMR 6628-MAPMO, Universit\'{e} d' Orl\'{e}ans, B. P.  6759, F-45067 Orl\'{e}ans Cedex, France} \email{serguei.tcheremchantsev@univ-orleans.fr}

\def\eps{\epsilon}

\theoremstyle{plain}
\newtheorem{theorem}{Theorem}[section]

\newtheorem{lemma}[theorem]{Lemma}

\theoremstyle{definition}
\newtheorem{definition}[theorem]{Definition}

\newtheorem{remark}[theorem]{Remark}

\begin{document}


\begin{abstract}
We develop tools to study arithmetically induced singular continuous
spectrum in the neighborhood of the arithmetic transition in the hyperbolic regime. This leads to first transition-capturing upper bounds on packing and
multifractal dimensions of spectral measures. We achieve it through the proof of partial localization of generalized eigenfunctions, another first result of its kind in the singular continuous regime. The proof is based also on a 
 general criterion for lower bounds on concentrations of Borel measures  as a corollary of boundary
behavior of their Borel-type transforms, that may be of wider use and independent interest.
\end{abstract}
\maketitle

\section{Introduction}

In the physics literature, the positivity of the Lyapunov exponent $L(E)$ (see \eqref{L}) is routinely considered as a signature of localization of one-dimensional Schr\"odinger operators. However it is well known that such positivity can well coexist with singular continuous spectrum. One-frequency quasiperiodic operators are operators on 
$   \ell^2(\mathbb{Z})$ of the form
 \begin{equation}\label{qp}
 (H_{v,\alpha,\theta}u)(n)=u({n+1})+u({n-1})+ v (\theta+n\alpha)u(n),  
 \end{equation}
where $v$ is a $1$-periodic real-valued function, $\alpha $ is the frequency, and $\theta $ is the phase.
For them, the transitions between pure point and singular continuous spectrum in the regime of positive Lyapunov exponents are known, in several prominent examples, e.g. \cite{ayz,jl1,simonmar,jlmar,jkach,gyz, gj}, to be sharp and, for a.e. phase, determined by the interplay of the positive Lyapunov exponent and an arithmetic
coefficient $\beta(\alpha)$, 

 \begin{equation}\label{Def.Beta}
  \beta= \beta(\alpha)=\limsup_{n\rightarrow\infty}\frac{\ln q_{n+1}}{q_n},
 \end{equation}
that measures the
exponential strength of frequency resonances. Here  $ \frac{p_n}{q_n} $  are  the continued fraction  approximants   of  $\alpha.$  Those sharp arithmetic transitions are expected to be universal for a large class of operators \eqref{qp} \cite{jl1,jcong22,gj,wm}. Indeed, under very weak conditions on the potential (e.g. $v$ of bounded variation is enough) 
there is no point spectrum supported on $\{E: L(E) <\beta(\alpha)\}$
\cite{jkach2}, while a.e. localization/point spectrum on $\{E: L(E) >\beta(\alpha)\}$ has been proved, besides the almost Mathieu operator \cite{ayz,jl1,gyz}, for all operators of type I \cite{gj} and all monotone potentials \cite{jkach}.

For the localization side, there are currently two approaches to sharp transition results: through the analysis of the eigenfunctions, first developed in \cite{jl1} and, for analytic potentials, through dual reducibility \cite{ayz,gyz}. Both provide a very simple and satisfactory picture of why the transition happens where it does: the exponential decay of the eigenfunctions deteriorates as the parameters approach the transition, and the norms of reducing matrices explode with the transition approach in the reducibility-based proofs.

No description like that exists for the singular continuous spectrum. The Gordon-based proofs of absence of point spectrum below the transition show just that: that there are no decaying solutions, not distinguishing in any way between the situations far from the transition point and near to it. Particularly, there have been no quantity that would capture the continuity of the singular continuous spectrum and known to tend to zero as the transition is approached.

Taking the prototypical 1D
quasiperiodic model, the almost Mathieu operator 
 \begin{equation}\label{Def.AMO}
 (H_{\lambda,\alpha,\theta}u)(n)=u({n+1})+u({n-1})+ 2\lambda \cos 2\pi (\theta+n\alpha)u(n),  
 \end{equation}
where $\lambda>0$ is called the coupling, 
 as a case in point, the transition, for arithmetically explicit a.e. phase,  does happen at $L=\beta(\alpha)$, where $L=\ln\lambda$ is the Lyapunov exponent, which is constant in $E$ on
the spectrum \cite{bj}.
Since for $\beta(\alpha)<L=\ln\lambda$ there is no singular
continuous spectrum at all, it is natural to expect that for
$\beta(\alpha)$ just above $L=\ln\lambda$ the singular continuous spectrum is
``more singular" than for higher values $\beta(\alpha),$ and, more
generally, that there is a way to capture the sharp transition through some quantitative property of the singular continuous spectra depending on the resonance
strength quantified by $\beta(\alpha)$.

The most standard candidate to distinguish between different
kinds of singular continuous spectra is Hausdorff dimension. However,
in the regime of positive Lyapunov exponents, Hausdorff dimension is always $0,$ at least for a.e. phase (proved by B. Simon \cite{simpot} as a corollary of the potential theory) and often for all phases (see a series of results, starting with \cite{jlast1} and with recent advances in \cite{jp,liu,ss,lpw}). Intuitively, Hausdorff dimension is associated with $\liminf$\footnote{Note that the denominator is negative, so $liminf$ measures the {\it largest} power-law of the concentration along a subsequence.} of the power-law behavior of the concentration function of the measure. The zero upper bound, using the techniques of \cite{jlast1,jlast2}, can be achieved through controlled occasional largeness of transfer matrices, which, not surprisingly, does happen in the regime of positive Lyapunov exponents. 

In \cite{jz} the
authors showed that dimensions associated with $\limsup$ of the power-law behavior of the concentration function, e.g. the packing dimension, can be expected to provide meaningful information in the positive Lyapunov exponents regime. In particular, for quasiperiodic operators \eqref{qp} packing dimension of any spectral measure $\mu$ satisfies ${\rm dim}^+_P(\mu)> 1-C\frac L\beta,$ where ${\rm dim}^+_P$ is defined in Section \ref{Fractal Dimension Defs}. As a corollary, there is a sharp arithmetic result for analytic potentials at the ``most continuous" end of the scale: in the regime of positive Lyapunov exponents, spectral measures of operators \eqref{qp} with analytic potentials $v$ have spectral dimension\footnote{A $\limsup$-based quantity that is always smaller than packing dimension. The highest dimension is $1.$} $1$ if and only if $\beta(\alpha)=\infty.$ However, the singularity bound of \cite{jz} also depended on the basic transfer-matrix estimates only, and, while effective at the most continuous end, was too crude to capture the spectral transition. Indeed, following the proof of Theorem 5 in \cite{jz}, one can see that, in the regime $L\sim\beta$ the upper bound on spectral dimension in Theorem 5 takes the form $\frac1{1+\frac{cL^2}{\beta}}$, thus staying bounded away from zero as $\beta\to L.$

Here we develop tools, based on estimates of generalized eigenfunctions, to prove upper bounds in the regime of positive Lyapunov exponents that do capture the transition from the singular continuous side. 

We formulate our main result for the almost Mathieu operator, however, we believe that the same formulation with $\ln \lambda$ replaced by $\min L(E)$ should work for the other cases where sharp arithmetic transition has been established, as should the key elements of our technique.

For any irrational  $\alpha$, we say that  phase $\theta\in (0,1)$   is
 $\alpha$-Diophantine, if there exist $ \kappa>0$ and $\nu>0$ such that
 \begin{equation}\label{DCtheta}
   ||2\theta-k\alpha||_{\mathbb{R}/\mathbb{Z}} \geq \frac{\kappa}{|k|^{\nu}},
 \end{equation}
 for any $k\in \mathbb{Z} \backslash \{0\}
 $, where $||x||_{\mathbb{R}/\mathbb{Z}}={\rm dist}(x,\mathbb{Z})$.
 Clearly, for any irrational number $\alpha$, the set of 
$\alpha$-Diophantine phases is of full Lebesgue measure.
 \par
  Recall that each  vector $\varphi\in \ell^2(\mathbb{Z})$ has a unique corresponding spectral measure $\mu_{\varphi}$, defined as
\begin{equation}\label{G.defspectralmeasure}
    ((H-z)^{-1}\varphi,\varphi)=\int\frac{d\mu_{\varphi}(x) }{x-z}.
\end{equation}
Let ${\rm dim}^+_P(\mu)$ denote the upper packing dimension of a Borel measure $\mu,$ see Section \ref{Fractal Dimension Defs}.
\begin{thm}\label{Maintheoremp}
For $\alpha\in \mathbb{R}\backslash \mathbb{Q}$,  
the following holds for the almost Mathieu operator \eqref{Def.AMO}, any $\alpha$-Diophantine phase  $\theta $ and any $\varphi\in \ell^2(\mathbb{Z})$,
 \begin{description}
 \item[Case 1] if $\lambda\ge e^{\beta(\alpha)}$,    then  ${\rm dim}^+_P(\mu_{\varphi })=0$.
 \item[Case 2]if 
 $\lambda< e^{\beta(\alpha)}$,   then     ${\rm dim}^+_P(\mu_{\varphi })\leq 2(1-\frac{\ln\lambda}{\beta(\alpha)})$.

 \end{description}
\end{thm}
\begin{remark}
	\begin{itemize}
	 
		\item If $\lambda>e^{\beta(\alpha)}$ and $\theta$ is Diophantine with respect to $\alpha$,   the spectral measure  $\mu_{\varphi}$  is supported  on a countable set (pure point spectrum)  \cite{jl1} so only the result for $\lambda=e^{\beta(\alpha)}$ is new in Case 1. In this case singular continuous spectrum does occur depending on further arithmetic properties of $\alpha$ \cite{ajz}.
		\item If $0<\lambda<1$, then for all $(\alpha,\theta)\in \R ^2$,
		the spectral measure $\mu_{\varphi}$  is purely absolutely continuous \cite{avila2008absolutely} so of packing dimension $1$. We do not address the packing dimension at $\lambda=1,$ where the spectrum is singular continuous in absence of hyperbolicity.
  \item The result of Case 2 is only meaningful for  $\lambda>e^{\frac{1}{2}\beta(\alpha)}$ so, unlike the one of \cite{jz}, our estimate does not say anything for $\beta(\alpha)\gg L$, also suggestive that the factor of $2$ may be not sharp.
	\end{itemize}
\end{remark}

A finer popular characteristic of singular continuous measures is its family of multifractal dimensions $D^\pm_\mu(q),$ see Section \ref{Fractal Dimension Defs}. It turns out these dimensions also capture the arithmetic transition.

\begin{thm}\label{Maintheoremp1}
	For $\alpha\in \mathbb{R}\backslash \mathbb{Q}$,  
the following holds for the almost Mathieu operator \eqref{Def.AMO}, any $\alpha$-Diophantine phase  $\theta $,  any initial vector $\varphi$,
	and     $ q \ge  { \frac{3} {2}}$, 
	$$
	D_{\mu_{\varphi}}^+(q) \le \frac{2\beta(\alpha) -2\ln \lambda}{2\beta(\alpha) -\ln \lambda}.
	$$

\end{thm}

An important part of our proof is a general result allowing to estimate packing and multifractal dimensions of Borel measures on $\mathbb{R}$ from {\it above} through the boundary behavior of their Borel-type transforms.

Power-law subordinacy theory \cite{jlast1,jlast2} has linked \footnote{The so-called Jitomirskaya-Last inequality} certain properties of formal solutions to $Hu=Eu$ of general one-dimensional operators
\begin{equation}\label{1d}
 (Hu)(n)=u({n+1})+u({n-1})+ v_n u(n),  
 \end{equation}
to the boundary behavior of the $m$-function, which is the Borel transform of a certain spectral measure. The general link between dimensions and boundary behavior of Borel transforms of Borel measures goes back to \cite{djls}. It is obvious that
\begin{equation}\label{borel}
    \mu([x-\epsilon,x+\epsilon])\leq 2\epsilon \Im \int_\mathbb{R} \frac{d\mu(y)}{x-y-i\epsilon}.
\end{equation}
This enables immediate bounds of Hausdorff dimension (defined via $\liminf {{\rm log} \mu ([x-\eps, x+\eps]) \over {\rm log} \eps}$) through the behavior of the $m$-function, and therefore properties of solutions. However, it is trickier to bound from above the $\limsup{{\rm log} \mu ([x-\eps, x+\eps]) \over {\rm log} \eps}$ \footnote{Since $\log \epsilon <0$ this is equivalent to bounding $\mu([x-\epsilon,x+\epsilon])$ from below} responsible for the packing dimension; that's why the authors of \cite{jz} for their upper bounds dealt with so-called spectral dimension defined directly through the boundary behavior of the $m$-function rather than the asymptotics of the concentration function.

Here we provide a general criterion allowing, in particular, to bound packing dimension of a measure from above (or concentration from below) through the boundary behavior of the Borel transform, that was previously missing.\footnote{As we were finalizing this work, a different approach to bounding $\mu([x-\epsilon,x+\epsilon])$ from below from the properties of Borel transform was developed, in a different context in \cite{lyz}. There the authors obtained a precise universal asymptotic of concentration functions in the absolutely continuous regime, a very delicate result. Their lower bound on concentration did require certain further input.} This comes as a particular case of a  family of results for new Borel-type transforms, that may be of independent interest.

Since  point spectrum part of a Borel measure always has fractal dimension zero, we will assume a Borel measure 
$\mu$ is continuous. We also assume $\mu$ is bounded, $\mu(\R)<\infty$.



 Let $x \in {\rm supp} \mu$.
Define the local exponents of the concentration function of $\mu$:

\begin{equation}\label{gamma-}
 \gamma^-_{\mu}(x)=\liminf_{\eps \to 0+}  {{\rm log} \mu ([x-\eps, x+\eps]) \over {\rm log} \eps},
\end{equation}
and
\begin{equation}\label{gamma+}
 \gamma^+_{\mu}(x)=\limsup_{\eps \to 0+}  {{\rm log} \mu ([x-\eps, x+\eps]) \over {\rm log} \eps}.
\end{equation}

Assume  $m>0$.   We define a new family of transformations, that we call $m$-Borel transforms by
\begin{equation}\label{mera04}
J_{\mu,m}(x,\eps) := \eps^{m} \int_{\R} {d \mu(y) \over |x-y|^m +\eps^m}.
\end{equation}
Note that $J_{\mu,2}(x,\eps)=\epsilon\Im \int_\mathbb{R} \frac{d\mu(y)}{x-y-i\epsilon},$ so we recover the imaginary part of the standard Borel transform.

\begin{thm}\label{D+}
		Let $m>\varsigma\geq 0$.
Assume that 
\begin{equation*}
   \liminf_{\eps\to 0+} \eps^{- \varsigma}J_{\mu,m}(x,\eps)>0.
 \end{equation*}
 Then
\begin{equation}\label{me14}
\gamma_\mu^+(x) \le {\varsigma (m-\gamma_\mu^-(x)) \over m-\varsigma}.
\end{equation}
In particular,
\begin{equation}\label{inme14}
\gamma_\mu^+(x) \le {m\varsigma   \over m-\varsigma}.
\end{equation}
\end{thm}

\begin{cor}\label{CD+}
	Let $m>\varsigma\geq 0$.
Assume that
	\begin{equation*}
	\liminf_{\eps\to 0+} \eps^{- \varsigma}J_{\mu,m}(x,\eps) >0
	\end{equation*}
	for almost every $x$ with respect to measure $\mu$. Then we have
	\begin{equation*}
	{\rm dim}^+_P(\mu) \le m \varsigma /(m-\varsigma).
	\end{equation*}
	Under the further assumption that   $\gamma_\mu^-(x) \ge \delta>0$ for almost every $x$ with respect to $\mu$, we have that
			$$
		{\rm dim}^+_P(\mu) \le \varsigma (m-\delta)/(m-\varsigma).
		$$
\end{cor}

Bounds on  the multifractal dimensions, defined in Section 2.1, can also be obtained as corollaries.

\begin{cor}\label{newCD-}
{\sl
 
 	Let $m>\varsigma\geq 0$. Assume that there exists a set $A\subset \R$  with $\mu(A)>0$ such that  for  all $x\in A$, 
 \begin{equation*}
   \limsup_{\eps\to 0+} \eps^{- \varsigma}J_{\mu,m}(x,\eps) >0.
 \end{equation*}
 Then for any $q>1$
$$
D_\mu^-(q)  \le   \varsigma.
$$
}\end{cor}
\begin{cor}\label{newCD+}
		Let $m>\varsigma\geq 0$.
 Assume that there exists a set $A\subset \R$  with $\mu(A)>0$ such that  for  all $x\in A$, 
\begin{equation*}
   \liminf_{\eps\to 0+} \eps^{- \varsigma}J_{\mu,m}(x,\eps)>0. 
 \end{equation*}
Then for any $q>1$
$$
D_\mu^+(q)  \le   m \varsigma /(m-\varsigma).
$$
Under the further assumption that $\gamma_\mu^-(x) \ge \delta>0$ for   all $x\in A$, we have that 
for any $q>1$
$$
D_\mu^+(q)    \le \varsigma (m-\delta)/(m-\varsigma).
$$
\end{cor}
One can also prove a better upper bound for $D_\mu^+(q)$ for  large $q$.
\begin{thm}\label{upD+}
  Let $m>1>\varsigma> 0$.
Assume that there exists a set $A\subset \R$  with $\mu(A)>0$ such that  for  all $x\in A$, 
\begin{equation*}
\liminf_{\eps\to 0+} \eps^{- \varsigma}J_{\mu,m}(x,\eps)>0. 
\end{equation*} 
Then we have that  for all $q$ with  $ q \ge 1+{1 \over m}$, 
$$
D_\mu^+(q) \le \varsigma.
$$
 
\end{thm}

Finally, $m$-Borel transform leads also to a new bound for Hausdorff dimension.

Let   ${\rm dim}^+_H(\mu)$    be the upper Hausdorff   dimension of the measure $\mu$, see Section \ref{Fractal Dimension Defs}.

\begin{thm}\label{D-}
 
		Let $m>\varsigma\geq 0$.
Assume that
 \begin{equation}\label{newequa4}
   \limsup_{\eps\to 0+}\eps^{- \varsigma}J_{\mu,m}(x,\eps) >0.
 \end{equation}
  Then
  \begin{equation*}
    \gamma_{\mu}^-(x)\leq \varsigma.
  \end{equation*}
  
\end{thm}

\begin{cor}\label{CD-}
 	Let $m>\varsigma\geq 0$.
	Assume that
		\begin{equation*}
		\limsup_{\eps\to 0+} \eps^{- \varsigma}J_{\mu,m}(x,\eps)>0
		\end{equation*}
		for almost every $x$ with respect to measure $\mu$.
		Then
		$$
		{\rm dim}^+_H(\mu) \le \varsigma.
		$$
\end{cor}

\section{Preliminaries}
\subsection{Fractal dimensions of measures and proof of corollaries}\label{Fractal Dimension Defs}

We first define the Hausdorff dimension ${\rm dim}_H(S)$. Fix \(S \subset \R\) and \(\gamma \in [0,1].\)
We define
 \[H_{\gamma,\delta}(S) = \inf \left\{\sum_{j = 1}^\infty |B_j|^\gamma : |B_j| < \delta; E \subset \bigcup_{j=1}^\infty B_j \right\},\]
 where the \(\inf\) is over all \(\delta\)-covers by intervals \(B_j\) of diameter at most \(\delta.\)

Set
 \[h^\gamma(S) = \lim_{\delta\to 0^+} H_{\gamma,\delta}(S).\]
 
$h^\gamma(S)$ is a Borel measure called \(\gamma\)\textit{-dimensional Hausdorff measure}.
There exists a unique $\gamma_0$ such that
 \[\gamma_0 = \inf\{\gamma: h^\gamma(S) = 0\}\] and \[\gamma_0 = \sup\{\gamma: h^\gamma(S) = \infty\}.\]
 We say that \(\gamma_0\) is the \textit{Hausdorff dimension} of \(S.\)

We now define the packing dimension of a Borel set \( S \), \(\text{dim}_{P}(S)\). A \(\delta\)-packing of an arbitrary set \( D \subset \mathbb{R} \) is defined as a countable collection \(\{B(E_k, r_k)\}_{k \in \mathbb{N}}\) of disjoint closed intervals centered at \( E_k \in D \) with \( r_k < \frac{\delta}{2} \).

The quantity \( P^{\gamma}_{\delta}(D) \) is defined by:
\[
P^{\gamma}_{\delta}(D) = \sup \left\{ \sum_{k \in \mathbb{N}} (2r_k)^{\gamma} \mid \{B(E_k, r_k)\}_{k \in \mathbb{N}} \text{ is a } \delta\text{-packing of } D \right\}.
\]

The \(\gamma\)-packing measure of a Borel set is defined in two steps:
\[
\tilde{P}^{\gamma}(D) = \lim_{\delta \to 0} P^{\gamma}_{\delta}(D),
\]
\[
P^{\gamma}(D) = \inf \left\{ \sum_{n \in \mathbb{N}} \tilde{P}^{\gamma}(D_n) \mid D_n \text{ is Borel, } \cup_n D_n = D \right\}.
\]

For any Borel set \( S \subset \mathbb{R} \), there exists a unique \(\text{dim}_{P}(S) \in [0, 1]\) such that \( P^{\gamma}(S) = 0 \) for any \( \gamma > \text{dim}_{P}(S) \) and \( P^{\gamma}(S) = \infty \) for any \( \gamma < \text{dim}_{P}(S) \). This unique value \(\text{dim}_{P}(S)\) is called the packing dimension of the set \( S \).

Let $\mu$ be a Borel measure on $\R$.
The upper/lower Hausdorff    and upper/lower packing dimensions of the measure $\mu$ is defined as 
\begin{align}
\dim^+_H(\mu) &= \inf\{\dim_H(S): \mu(S) =\mu(\R)\} ,\\
\dim^+_P(\mu) &= \inf\{\dim_P(S): \mu(S) =\mu (\R)\},\\
\dim^-_H(\mu) &= \inf\{\dim_H(S): \mu(S) >0\} ,\\
\dim^-_P(\mu) &= \inf\{\dim_P(S): \mu(S) >0\}.
\end{align}

Let us now define the
generalized R\'{e}nyi dimensions
$D_\mu^\pm (q)$ of $\mu$. For any $\eps >0, q>0$ consider the sum
\begin{equation}\label{renyi}
S_\mu(q, \eps)=\sum_{j \in {\Z}} ( \mu([j \eps, (j+1) \eps))^q.
\end{equation}

The lower and upper multifractal dimensions can be defined as follows (this is one of many equivalent definitions):
\begin{equation*}
D_\mu^-(q)=\liminf_{\eps \to 0+}{ {\rm log} S_\mu(q, \eps) \over (q-1) {\rm log} \eps},
\end{equation*}
 and
 \begin{equation*}
D_\mu^+(q)=\limsup_{\eps \to 0+}{ {\rm log} S_\mu(q, \eps) \over (q-1) {\rm log} \eps}.
\end{equation*}


\begin{thm}\cite{gua} \label{HDIM AltDef}
\begin{align}
\dim_H^+(\mu) &= \mu-ess \sup \gamma_\mu^-(x) \\
\dim_P^+(\mu) &= \mu -ess \sup \gamma_\mu^+(x)\\
\dim_H^-(\mu) &= \mu-ess \inf \gamma_\mu^-(x) \\
\dim_P^-(\mu) &= \mu -ess \inf \gamma_\mu^+(x).
\end{align}
\end{thm}
\begin{thm}\cite[Prop.4.1]{BGTJMPA01} \label{DHP}
For any $q>1$,
\begin{align}
D_\mu^-(q) &\leq \dim_H^-(\mu) \\
D_\mu^+(q) &\leq \dim_P^-(\mu).
\end{align}
\end{thm}
\begin{proof}[\bf Proof of Corollaries \ref{CD-} and \ref{CD+}]
Corollaries of \ref{CD-} and \ref{CD+} follow from Theorems \ref{D+}, \ref{D-} and \ref{HDIM AltDef}.
\end{proof}
\begin{proof}[\bf Proof of Corollaries   \ref{newCD-} and \ref{newCD+}]
Corollaries of \ref{newCD-} and \ref{newCD+} follow from Theorems \ref{D+}, \ref{D-}, \ref{HDIM AltDef} and \ref{DHP}.
\end{proof}
\subsection{$m$-functions and power-law subordinacy}
Let $\tilde {m}_j(z)$ \footnote{ We use the definitions in  \cite{jlast2}. We should mention that  in \cite{tes95}, the authors used   $-m_2$  as  the $m$ function of the  left half line problem.}, $j=1,2$ be the $m$ functions of the right and left half line problems with boundary condition
\begin{equation}\label{mbd}
  u(0)\cos2\pi x_0+u(1)\sin2\pi x_0=0.
\end{equation}
Let ${m}_j(z)$, $j=1,2$ be the $m$ functions of the right and left half line problems with  the Dirichlet boundary condition, namely
\begin{equation}\label{mbdd}
  u(0)=0.
\end{equation}

It is well known (c.f. (2.5) in \cite{jlast2} or (1.27) in \cite{tes95}) that
\begin{equation}\label{glast12}
  {m}_1=\frac{\tilde {m}_1\cos2\pi x_0-\sin2\pi x_0}{\tilde {m}_1\sin2\pi x_0+\cos2\pi x_0},
\end{equation}
and
\begin{equation}\label{glast13}
  {m}_2=\frac{\tilde {m}_2\cos2\pi x_0+\sin2\pi x_0}{\cos2\pi x_0-\tilde {m}_2\sin2\pi x_0}.
\end{equation}
Let $\mu_{\delta_0}$ and  $\mu_{\delta_1}$ be the spectral measures of vectors $\delta_0$ and $\delta_1$ respectively for operator $H_{\lambda,\alpha,\theta}$.  
Let $M_1$ and $M_2$  be the  Borel transformation of $\mu_{\delta_0}$ and $\mu_{\delta_1}$, that is
\begin{equation}\label{G.borel1}
    M_1(z)=\int_{\R} \frac{d\mu_{\delta_0}(x)}{x-z},
\end{equation}
and 
\begin{equation}\label{G.borel2}
    M_2(z)=\int_{\R} \frac{d\mu_{\delta_1}(x)}{x-z}.
\end{equation}
 
It is known that (c.f. (1.71) in \cite{tes95} )
\begin{equation}\label{glast14}
    M_{1}=-\frac{1}{{m}_1+{m}_2},
\end{equation}
and 
\begin{equation}\label{glast15}
    M_{2}=\frac{{m}_1 {m}_2}{{m}_1+{m}_2}.
\end{equation}

For any  function  $u$ on $\Z$, define $||u||_{L_1,L_2}^2$ as follows,
\begin{eqnarray*}
	||u||_{L_1,L_2}^2&=& \sum_{j=0} ^{[ L_1]}   |u(j)| ^2+ (L_1-[ L_1]) |u([ L_1]+1) |^2 \\
	& &+\sum_{j=1} ^{[ L_2]} |u(-j)| ^2+(L_2-[ L_2]) |u(-[ L_2]-1) |^2,
\end{eqnarray*}
for $L_1,L_2\geq 0$.

Denote  by $ u_{x}(\cdot,E)$   the solution of $Hu=Eu$    with  the initial  condition $u_{x}(0,E)=\sin 2\pi x$ and  $u_{x}(1,E)=-\cos 2\pi x$. 

By the constancy of the Wronskian, one has that for all $n\in\Z$,
\begin{equation}\label{gnov112}
 \det  \left(
	\begin{array}{cc}
		u_x(n+1,E)  & 	u_{x+1/4}(n+1,E)  \\
		u_x(n,E)& u_{x+1/4}(n,E)\\
	\end{array}
	\right) =1.
\end{equation}
The following theorem is essentially proved in \cite{kkl}. The precise formulation presented here is from \cite{dtjam} (see page 127).
\begin{thm}\label{kklthm}
    Denote  by 

\begin{eqnarray*}
  a(L) &=& ||u_{1/4+x_0}||^2_{L,0} ,\\
  b(L) &=& ||u_{x_0}||^2_{L,0}, \\
  \omega(L) 
    &=& \max_{x}||u_x(\cdot,E)||_{L,0}\cdot\min_{x}||u_{x}(\cdot,E)||_{L,0}.
\end{eqnarray*}

Let $L(\eps)$ be such that
\begin{equation}\label{G.Lnew}
  \omega(L(\eps))=\frac{1}{\eps}.
\end{equation}

Then we have that for an absolute constant  $C$,
\begin{equation*}
    \Im \tilde {m}_1(E+i\eps)\geq \frac{1}{C} \frac{1}{\eps} \frac{1}{b(L(\eps))}.
\end{equation*}
\end{thm}
\subsection {Lyapunov exponent}
Let
	\begin{equation}\label{G.transfer}
		A_{k}^E(\theta)=\prod_{j=k-1}^{0 }A^E(\theta+j\alpha)=A^E(\theta+(k-1)\alpha)A^E(\theta+(k-2)\alpha)\cdots A^E(\theta)
	\end{equation}
	and
	\begin{equation}\label{G.transfer1}
		A_{-k}^E(\theta)=(A^E_{k} (\theta-k\alpha))^{-1}
	\end{equation}
	for $k\geq 1$,
	where $A^E(\theta)=\left(
	\begin{array}{cc}
		E- v(\theta) & -1 \\
		1& 0\\
	\end{array}
	\right)
	$.
	$A_{k}^E$  is called the ($k$-step) transfer matrix for the quasi-periodic Schr\"odinger operator \eqref{qp}.

	The Lyapunov exponent for the quasi-periodic Schr\"odinger operator \eqref{qp} 
	is given  by
	\begin{equation}\label{L}
		L(E)=\lim_{|k|\rightarrow\infty} \frac{1}{k}\int_{\mathbb{R}/\mathbb{Z}} \ln \| A_k^E(\theta)\|d\theta.
	\end{equation}
 
\section{Proof of Theorem \ref{D-}}
As a warm-up, we start with the proof of Theorem \ref{D-} because some of the techniques that will play a role later appear here in an easier  version. 
The proof  follows the ideas of that of Lemma 3.3 in   \cite{djls}. In our terms they prove in the case $m=1$
that $\gamma_\mu^-(x) \le \varsigma$ under the assumption that $\gamma_\mu^-(x) \le 1=m$ (which holds for $\mu$-a.e. $x \in {\R}$.)
Since we prove our result for any single point $x$, we should first establish that $\gamma_\mu^-(x) \le m$ before proving that
$\gamma_\mu^-(x) \le \varsigma$. Even for $m=1$ our result is a bit more general, since it is possible that $1 <\gamma_\mu^-(x) \le  \varsigma$.
\begin{proof}[\bf Proof of Theorem \ref{D-}]
Let $\eps \in (0,1]$. For any integer $k \ge 0$, define
$$
A_k=\{ y \in {\R} : \ k \eps \le |x-y| < (k+1)\eps \},
$$
and
\begin{equation*}
a_k=\mu((x-k \eps, x+ k\eps)), k \ge 1; a_0=0.
\end{equation*}
Observe that $\mu (A_k)=a_{k+1}-a_k$ (noting that $\mu$ is continuous in the present paper). We can bound from above
\begin{align}
J_{\mu,m}(x,\eps) &= \sum_{k=0}^{\infty}\eps^m  \int_{A_k} {d \mu(y) \over |x-y|^m +\eps^m}\nonumber\\
&\le  \sum_{k=0}^{\infty} {1 \over k^m +1} (a_{k+1}-a_k).\label{newequa1}
\end{align}

Therefore, one has that
\begin{equation}\label{me12}
\sum_{k=0}^N { a_{k+1} -a_k \over k^m+1}=\sum_{k=1}^N a_k \left( {1 \over (k-1)^m +1} -{1 \over k^m +1} \right)+
{a_{N+1} \over N^m +1} .
\end{equation}
Since
\begin{equation}\label{newequa2}
{1 \over (k-1)^m +1} -{1 \over k^m +1} \sim \frac{C}{k^{m+1}}, \ k \to \infty,
\end{equation}
one has  that  $\sum_{k=1}^{\infty} a_k \left( \frac{1}{ (k-1)^m +1} -\frac{1} {k^m +1} \right)$ converges. 

By \eqref{newequa1}, \eqref{me12} and \eqref{newequa2},
we  have that
\begin{equation}\label{newequa6}
 J_{\mu,m}(x,\eps)\leq C\sum_{k=1}^{\infty}\frac{a_k}{k^{m+1}}.
\end{equation}

We will first show that $\gamma_\mu^-(x) \le m$ and next that $\gamma_\mu^-(x) \le \varsigma$.
 Suppose that $\gamma_\mu^-(x)>m$ (in particular, including $\gamma_\mu^-(x)=\infty$). Take any $\nu$ with $ 0<\nu<\gamma_\mu^-(x)-m$.
The definition of $\gamma_\mu^-(x)$ implies that for any $\eta \in (0,1]$,
\begin{equation}\label{me14d}
\mu([x-\eta, x+\eta]) \le C \eta^{m+\nu}.
\end{equation}

It follows from (\ref{me14d}) that for any positive integer $k$,
\begin{equation}\label{newequa7}
a_k=\mu ((x-k \eps, x+k \eps)) \le C k^{m+\nu}  \eps^{m+\nu} .
\end{equation}

Let $N=[\frac{1}{\epsilon}]$.  By \eqref{newequa6} and \eqref{newequa7},  we have
\begin{align}
J_{\mu,m}(x,\epsilon)
& \leq C \left(\sum_{k=1}^N+\sum_{k=N+1}^{\infty}\right) {a_k \over k^{m+1}} \nonumber\\
&\leq   C\sum_{k=1}^N {\eps^{m+\nu} k^{m+\nu} \over k^{m+1}} +
C \sum_{k=N+1}^{\infty} {1 \over k^{m+1}} \nonumber\\
& \leq C \epsilon^{m+\nu} N^{\nu}+CN^{-m}\nonumber \\
&\leq C\epsilon^m.
\end{align}

It implies 
\begin{equation}\label{newequa3}
\eps^{-\varsigma}J_{\mu,m}(x,\epsilon)\leq C \epsilon^{m-\varsigma}.
\end{equation}
Since $m>\varsigma$,  \eqref{newequa3} is in contradiction with  \eqref{newequa4}.
We conclude that $\gamma_\mu^-(x) \le m$.

If $\gamma_\mu^-(x)=0$,  Theorem \ref{D-}  has nothing to be proved.   Assume that   $\gamma_\mu^-(x)>\varsigma$. 
Let  $\delta$ be such that $\varsigma<\delta<\gamma_\mu^-(x)$.  By the definition of $\gamma_\mu^-(x)$, one has  that 
$$
\mu([x-\eta, x+\eta]) \le C \eta^{\delta},
$$
and hence
\begin{equation}\label{newequa8}
a_k=\mu ((x-k \eps, x+k \eps)) \le C k^{\delta}  \eps^{\delta} .
\end{equation}
By \eqref{newequa6} and \eqref{newequa8}, one has
\begin{align*}
J_{\mu,m}(x,\epsilon)
&\leq   C\sum_{k=1}^{\infty} \epsilon^{\delta}k^{\delta-m-1}\\
& \leq C\epsilon^{\delta},
\end{align*}
where the second inequality holds because 
 $\delta<\gamma_\mu^-(x) \le m$.
It implies 
\begin{equation}\label{newequa5}
\eps^{-\varsigma}J_{\mu,m}(x,\epsilon)\leq C \epsilon^{\delta-\varsigma}.
\end{equation}
Since $\delta>\varsigma$,  \eqref{newequa5} contradicts  \eqref{newequa4}.
\end{proof}


\section{Proof of Theorem \ref{D+}}
\begin{proof}

	{\bf Case 1: $\gamma_\mu^-(x)=0$}
	
	By  assumption, there is a constant $C>0$ such that 
	\begin{equation}\label{me8}
	\eps^m \int_{\R} {d \mu(y) \over |x-y|^m +\eps^m} \ge \frac{1}{C} \eps^\varsigma,
	\end{equation}
	for any $\epsilon>0$.
	
	Define two complementary sets:
	$$
	A=\{ y \in {\R} : |x-y| \le \eta \}, \ \ B=\{ y \in {\R} : |x-y| > \eta \}
	$$
	with some  $\eta>0$. Since  $\mu(\R)<\infty$, we  have that
	\begin{equation}\label{newequa10}
	\eps^m \int_B {d \mu(y) \over |x-y|^m +\eps^m} \le \eps^m \eta^{-m} \int_B d \mu (y) \leq C (\eps/ \eta)^m.
	\end{equation}

For given $\eta>0$, let us choose $\epsilon$ so that 
	\begin{equation}\label{newequa11}
C  (\eps/\eta)^m =\frac{1}{C_1}\eps^\varsigma,
	\end{equation}
	where $C_1$ is a constant which  is  larger than $2C$.
	
	By \eqref{me8}, \eqref{newequa10} and \eqref{newequa11}, we have
	\begin{align}
	\eps^m \int_A {d \mu(y) \over |x-y|^m +\eps^m} &\geq 	\eps^m \int_{\R} {d \mu(y) \over |x-y|^m +\eps^m} -	\eps^m \int_B {d \mu(y) \over |x-y|^m +\eps^m} \nonumber\\
	&\geq \frac{1}{2C}\epsilon^{\varsigma}. \label{newequa12}
	\end{align}

	On the other hand,
	\begin{equation}\label{me9}
	\eps^m \int_A {d \mu(y) \over |x-y|^m +\eps^m} \le  \int_A d \mu (y)= \mu([x-\eta, x+\eta]).
	\end{equation}
	It follows from \eqref{newequa11}, (\ref{newequa12}) and  (\ref{me9}), that
	\begin{equation}\label{me11}
	\mu([x-\eta, x+\eta]) \ge \frac{1}{C} \eta^{m \varsigma \over m-\varsigma}.
	\end{equation}
	This completes the proof for this case.
	
{\bf Case 2: $\gamma_\mu^-(x)>0$}	

 Let $N=[\eta/\eps]$ and
$$
B=\{ y \in {\R} \ : |x-y|>\eta \}, \ D=\{  y \in {\R} \ : |x-y|>N \eps \}.
$$
It is clear that
$B\subset D$. Similarly to \eqref{newequa6}, one has
\begin{align}
\eps^m \int_B {d \mu(y) \over |x-y|^m +\eps^m} &\le \eps^m \int_D {d \mu(y) \over |x-y|^m +\eps^m}\nonumber \\
&\le C \sum_{k=N}^{\infty} {a_k \over k^{m+1}}.\label{newequa13}
\end{align}

Let $\delta \in (0, \gamma_\mu^-(x))$.  By \eqref{newequa8} and \eqref{newequa13}, one has that
\begin{align}
\eps^m \int_B {d \mu(y) \over |x-y|^m +\eps^m}& \leq C\epsilon^{\delta}\sum_{k=N}^{\infty}k^{\delta-m-1} \nonumber\\
& \leq C \epsilon^{\delta} N^{\delta-m}\nonumber\\
&\leq C \epsilon^{m}\eta^{\delta-m}.
\end{align}

Fix $\eta$  and choose $\epsilon$ such that  $C \eps^m \eta^{\delta -m}=\frac{1}{C_1} \eps^{\varsigma}$. Arguing as in the proof of  Case 1, we obtain
$$
\mu([x-\eta, x+\eta]) \ge \frac{1}{C} \eta^\rho,
$$
where $\rho=\varsigma (m-\delta) /(m-\varsigma)$. Letting $\delta\to \gamma_\mu^-(x)$, we complete the proof.

\end{proof}
\section{ Proof of Theorem \ref{upD+}}

\begin{proof} 
Let $m>1, p \in (1/m, 1]$. 
	Set 
	\begin{equation*}
	f(x)=	\liminf_{\eps\to 0+}  \eps^{-\varsigma}J_{\mu,m}(x,\eps).
	\end{equation*}
	By assumption, one has 
	\begin{equation*}
f(x) >0
	\end{equation*}
	for  all $x\in A$. 
	Let $$ I(\eps):=\int_{\R} d \mu (x) J^p_{\mu,m}(x, \eps) \ge \int_A d \mu (x) J_{\mu,m}^p(x, \eps) .$$
	By Fatou's lemma,
\begin{equation}\label{me18}
\liminf_{\epsilon\to 0+}  I(\eps) \eps^{-p \varsigma}>\frac{1}{C}.
\end{equation}

Let us now bound  $I(\eps)$ from above. Define $I_j=[j \eps, (j+1)\eps)$. Then

\begin{equation}\label{me19}
I(\eps)=\sum_{j \in {\Z}} \int_{I_j}d \mu (x) \left( \sum_{k \in {\Z}} \int_{I_k} d \mu (y)
{ \eps^m \over |x-y|^m +\eps^m } \right) ^p.
\end{equation}
It is clear that for any $x \in I_j, \ y \in I_k$ one has $|x-y| \ge \eps (|k-j|-1)$ if $|k-j| \ge 2$. Define
$$
s(l)=0, |l| \le 1;  \ \ s(l)=|l|-1, \ |l| \ge 2.
$$
Thus, $|x-y| \ge \eps s(k-j)$ for all $x \in I_j, \ y \in I_k$.
It follows from (\ref{me19}) that
\begin{align}
    I(\eps) \le &\sum_{j \in {\Z}} \int_{I_j} d \mu (x) \left( \sum_{k \in {\Z}} \int_{I_k} {d \mu (y) \over s^m(k-j) +1} \right) ^p\nonumber \\
    =&\sum_j b_j \left( \sum_k {b_k \over s^m(k-j)+1} \right) ^p,
\label{me20}\end{align}
where $b_j=\mu (I_j)$. Using the elementary bound $(\sum_k t_k)^p \le \sum t_k^p, \ p \in (0,1] $ (with any number of terms in the sum), one has that
\begin{align}
I(\eps)  
&\leq
\sum_j b_j \sum_i {b_{i+j}^p \over (s^m(i)+1)^p }\nonumber\\
&=\sum_i {1 \over (s^m(i)+1)^p} \sum_j b_j b_{i+j}^p\nonumber\\
& = \sum_i {1 \over (s^m(i)+1)^p} h_i\label{newequa16},
\end{align}
where 
 $h_i=\sum_j b_j b_{i+j}^p$. 
Applying H\"older inequality with $\alpha=p+1, \alpha '=(p+1)/p$, we have that  for any $i$
\begin{align}
h_i &\le\left( \sum_j b_j^\alpha \right) ^{1/\alpha} \left( \sum_j b_{i+j}^{p \alpha'} \right) ^{1/\alpha '}\nonumber\\
&\leq 
\sum_j b_j^{p+1}\nonumber\\
& = S_\mu(p+1, \eps)\label{newequa17},
\end{align}
where $S_{\mu}$ is defined in \eqref{renyi}.
It is clear that 
\begin{equation}\label{me21}
 \sum_i {1 \over (s^m(i)+1)^p} <\infty
\end{equation}
since $mp>1$. 
Finally, by \eqref{newequa16}, \eqref{newequa17} and \eqref{me21}, we have
$$
I(\eps) \le C S_\mu (p+1, \eps),
$$
Together with (\ref{me18}) we obtain
$$
\liminf_{\eps \to 0+}   \eps^{-p \varsigma} S_\mu(p+1, \eps) \ge \frac{1}{C}.
$$
This  implies that $D_\mu ^+(p+1) \le \varsigma$. Since it is true for
any $p \in (1/m, 1]$ and $D_\mu ^+(q)$ is a continuous decreasing
function of $q$ (Recall that    $D_\mu^+(q)$ is non-increasing and continuous with respect to  $q$ \cite{BGTJMPA01,BGTDUKE01}), we obtain that $D_\mu ^+(q) \le \varsigma$ for all $q \ge 1+1/m$.
 
\end{proof}

 \section{ Basics for applications to the almost Mathieu operator }
Here we present the more detailed technical statements from \cite{jl1} that are needed for our proof.

We call $E$ generalized eigenvalue with the corresponding generalized eigenfunction  $u$   if $Hu=Eu,$ $u$ is nontrivial,
and
\begin{equation}\label{gen}
  |u(n)|\leq C(1+|n|).
\end{equation}

For  readers' convenience, we summarize the key assumptions and notations below: 
\begin{itemize}
\item  Frequency  $\alpha$ is irrational.  Let  $ \frac{p_n}{q_n} $  be  the continued fraction  approximants   of  $\alpha$ and $\beta(\alpha)$ be given by \eqref{Def.Beta}.
    \item 
 The statement of the theorem is only nontrivial if frequency  $\alpha$  satisfies  $ 0<\beta(\alpha) <\infty$, so we will assume that.
 \item $1<\lambda\leq e^{\beta}$.
 \item $\theta$ is Diophantine with respect to $\alpha$.
\item   $E$ is a generalized  eigenvalue.
\item  $\phi$ is the normalized generalized eigenfunction, namely $\phi$  satisfies  \eqref{gen} and $\phi^2(0)+\phi^2(1)=1$.
\item $\varepsilon$ is sufficiently small.
\end{itemize}

For simplicity, sometimes we omit the dependence on  parameters $\alpha,\lambda,\theta,E$.
  \par
Denote by
$$ P_k(\theta)=\det(R_{[0,k-1]}(H_{\lambda,\alpha,\theta}-E) R_{[0,k-1]}).$$

Note that, (see e.g. (32) in \cite{jl1}) by compactness, subadditivity, and the form of $A_k^E(\theta), $
for any $\varepsilon>0$,
\begin{equation}\label{Numerator}
    | P_k(\theta)|\leq e^{(\ln \lambda+\varepsilon)k}
\end{equation}
for $k$ large enough.



For interval $J\subset \mathbb{Z},$ let 
 $$G_J(x,y)=(R_{J}(H_{\lambda,\alpha,\theta}-E) R_{J})^{-1}(x,y) $$
 be the matrix elements of the Green function of the restriction.
    By Cramer's rule,  for given  $x_1$ and $x_2=x_1+k-1$, with
     $ y\in J=[x_1,x_2] \subset \mathbb{Z}$,  one has
     \begin{eqnarray}
       |G_J(x_1,y)| &=&  \left| \frac{P_{x_2-y}(\theta+(y+1)\alpha)}{P_{k}(\theta+x_1\alpha)}\right|,\label{Cramer1}\\
       |G_J(y,x_2)| &=&\left|\frac{P_{y-x_1}(\theta+x_1\alpha)}{P_{k}(\theta+x_1\alpha)} \right|.\label{Cramer2}
     \end{eqnarray}

\begin{definition}\label{Def.Regular}
Fix $t > 0$. A point $y\in\mathbb{Z}$ will be called $(t,k)$ regular if there exists an
interval $[x_1,x_2]$  containing $y$, where $x_2=x_1+k-1$, such that
\begin{equation*}
  | G_{[x_1,x_2]}(y,x_i)|\leq e^{-t|y-x_i|} \text{ and } |y-x_i|\geq  \frac{1}{5} k \text{ for }i=1,2.
\end{equation*}
\end{definition}
It is  easy to check  (e.g. p. 61, \cite{bgreen}) that for any eigen-equation $H\varphi=E\varphi$,
 \begin{equation}\label{Block}
   \varphi(x)= -G_{[x_1 ,x_2]}(x_1,x ) \varphi(x_1-1)-G_{[x_1 ,x_2]}(x,x_2) \varphi(x_2+1),
 \end{equation}
 where  $ x\in J=[x_1,x_2] \subset \mathbb{Z}$.

Given the eigen-equation $H\varphi=E\varphi$, 
 one has that  for large $|k_1-k_2|$ (e.g. \cite{liures}),
	\begin{equation}\label{g500}
		\left \|	\left(\begin{array}{c}
			\varphi(k_1+1) \\
			\varphi(k_1)                                                                                        \end{array}\right)\right\|
		\leq  e^{(\ln \lambda+\varepsilon) |k_1-k_2|}\left\|
		\left(\begin{array}{c}
			\varphi(k_2+1) \\
			\varphi(k_2)                                                                                        \end{array}\right)\right\|.
	\end{equation}

       \begin{definition}
     We  say that the set $\{\theta_1, \cdots ,\theta_{k+1}\}$ is $ \epsilon$-uniform if
      \begin{equation}\label{Def.Uniform}
        \max_{ x\in[-1,1]}\max_{i=1,\cdots,k+1}\prod_{ j=1 , j\neq i }^{k+1}\frac{|x-\cos2\pi\theta_j|}
        {|\cos2\pi\theta_i-\cos2\pi\theta_j|}<e^{k\epsilon}.
      \end{equation}
     \end{definition}
      Let $A_{k,r}=\{\theta\in\mathbb{R} \;|\;P_k(    \theta -\frac{1}{2}(k-1)\alpha  )|\leq e^{(k+1)r}\} $ with $k\in \mathbb{N}$ and $r>0$.
     We have the following lemma.
      \begin{lemma}\label{Le.Uniform} \cite[Lemma 9.3 ]{avila2009ten}
      Suppose  $\{\theta_1, \cdots ,\theta_{k+1}\}$ is  $ \epsilon_1$-uniform. Then there exists some $\theta_i$ in set  $\{\theta_1, \cdots ,\theta_{k+1}\}$ such that
     $\theta_i\notin A_{k,\ln\lambda-\epsilon}$ if    $ \epsilon>\epsilon_1$ and $ k$
      is sufficiently large.
      \end{lemma}
     We will apply these results in the following range.
    Fix   $t_2$ such that  $\frac{9\beta-\ln \lambda}{9\beta}<t_2<1$.
 Let $\sigma>0$ be small, such that, in particular,
 \begin{equation}\label{glast9}
t_1:=\frac{\beta-\ln\lambda}{\beta}+\sigma<t_2<1.
 \end{equation}

     Define $b_n=  q_n^{t_2}$. For any $k>0$, 
        we will distinguish two cases with respect to $n$:
         \par
        (i)   $|k-\ell q_n|\leq b_n$ for some $\ell\geq 0$,  called $n$-resonance.
          \par
        (ii)     $|k-\ell q_n|> b_n$ for all $\ell\geq0$, called  $n$-nonresonance.
\par
For the  $n$-nonresonant $y$,  let $n_0$ be the least positive integer such that $4q_{n-n_0}\leq {\rm dist}(y,  q_n \mathbb{Z})$.
 Let $s$ be the
largest positive integer such that $4sq_{n-n_0}\leq {\rm dist}(y,q_n \mathbb{Z}) $.

Set $I_1, I_2\subset \mathbb{Z}$ as follows
\begin{eqnarray*}
  I_1 &=& [-sq_{n-n_0},sq_{n-n_0}-1], \\
   I_2 &=& [ y-sq_{n-n_0},y+sq_{n-n_0}-1 ],
\end{eqnarray*}
and let $\theta_j=\theta+j\alpha$ for $j\in I_1\cup I_2$. The set $\{\theta_j\}_{j\in I_1\cup I_2}$
consists of $4sq_{n-n_0}$ elements.
\par
\begin{theorem} \cite[Theorem 3.5]{jl1}\label{Th.Nonresonant}
Assume that $\theta  $ is Diophantine  with respect  to $\alpha$ and  $$ \ln\lambda+8\ln (s q_{n-n_0}/q_{n-n_0+1})/q_{n-n_0}>0.$$  Let $E$ be a generalized eigenvalue.  Suppose   either

i) $b_n\leq |y|< C b_{n+1}$

or

ii) $0\leq |y|< q_n$

is $n$-nonresonant, where $C>1$ is a fixed constant.
Then   we have  that for  large $n$,   $ y$ is $  (\ln\lambda+8\ln (s q_{n-n_0}/q_{n-n_0+1})/q_{n-n_0}-\varepsilon,4sq_{n-n_0}-1)$ regular.
\end{theorem}
\begin{remark}
    \begin{itemize}
  \item  By the definition of $s$ and the non-resonance of $y$, one has
  \begin{equation*}
    4(s+1)q_{n-n_0}\geq b_n.
  \end{equation*}
  Therefore, for large $n$, 
  \begin{eqnarray*}
    \ln\lambda+8\frac{\ln (s q_{n-n_0}/q_{n-n_0+1})}{q_{n-n_0}} &\geq& \ln\lambda+8\frac{\ln q_{n}^{t_2}}{q_{n-n_0}}-8\frac{\ln q_{n-n_0+1}}{q_{n-n_0}}-\varepsilon \\
     &\geq& \ln\lambda+8\frac{\ln q_{n-n_0+1}^{t_2}}{q_{n-n_0}}-8\frac{\ln q_{n-n_0+1}}{q_{n-n_0}}-\varepsilon \\
     &\geq& \ln \lambda-8(1-t_2)\beta-\varepsilon\\
      &\geq&  \ln\lambda/9-\varepsilon>0.
  \end{eqnarray*}
  \item Although Theorem \ref{Th.Nonresonant} has been stated in \cite{jl1}  under the assumption that $\lambda>e^{\beta(\alpha)}$, only  the fact that $\ln\lambda+8\ln (s q_{n-n_0}/q_{n-n_0+1})/q_{n-n_0}>0$ is used in the proof.
\end{itemize}
\end{remark}

\section{  Partial localization}\label{SPartiallocalization}
The key to our proof is that, even in the regime of singular continuous spectrum all generalized eigenfunctions decay exponentially at $k$ roughly in the range $q_n^{t_1}<k<q_n^{t_2},$ for sufficiently large $n,$ where  $t_1,t_2$ is as in \eqref{glast9}.
The worst points for decay estimates are the resonant ones. We have
\begin{theorem}\label{Th.resonant}
Suppose  $2q_n^2q_{n+1}^{t_1}<|k|<b_{n+1}$  is $n$-resonant. Let $\phi$ be a generalized eigenfunction. Then for large enough $n$,
\begin{equation}\label{G.Iteration1}
   |\phi(k)|\leq \exp\{-(\ln \lambda-(1-t_1)\beta-\varepsilon)|k|\}.
\end{equation}
\end{theorem}

\begin{remark}
By the definition of $t_1$, $\ln \lambda-(1-t_1)\beta-\varepsilon>0$ if $\varepsilon$ is small.

\end{remark}

Suppose $H\varphi=E\varphi$.
Let $r_j^{\varphi}= \sup_{|r|\leq 10\varepsilon }|\varphi(jq_n+rq_n)|$.
For simplicity, let  $r_j= r_j^{\phi}$, where $\phi$ is the generalized eigenfunction.
Note that,    
we always  assume $\varepsilon>0$ is small enough and $n$ is large enough (may depend on $E,\sigma,\varepsilon$) below.

 \begin{lemma}\label{Le.resonant1}
Let  $k\in [jq_n,(j+1)q_n]$  and   $d= {\rm dist}(k,q_n \mathbb{Z})\geq     10\varepsilon q_n$.

  Suppose  either

i)  $|j|\leq \frac{100b_{n+1}}{q_n}$ and $b_{n+1}\geq \frac{q_n}{2}$,

  or

  ii) $j=0$,
  
  \par
then
\begin{equation}\label{Intervalk}
    |\varphi(k)|\leq \exp\{-(\ln \lambda- 2\varepsilon)(d-3\varepsilon q_n)\} \max\{ r_j^{\varphi},r_{j+1}^{\varphi}\}.
\end{equation}
\end{lemma}
\begin{proof}
The Lemma can be proved directly from   Theorem \ref{Th.Nonresonant}, see \cite[Lemma 4.1]{jl1} for details.
\end{proof}


By the definition of resonance and assumption of Theorem \ref{Th.resonant}, there exist $\ell\neq 0$ and $r$ such that $k=\ell q_n+r$ with $|r|\leq b_n$ and $q_n q_{n+1}^{t_1} \leq|\ell|\leq 2\frac{b_{n+1}}{q_n}$.
We may assume $\ell>0$ for simplicity.
For $[\frac{\ell}{q_n}]  \leq   |j| \leq 2\ell $,
set $I_1, I_2\subset \mathbb{Z}$ as follows
\begin{eqnarray*}
  I_1 &=& [-[\frac{1}{2}q_n], q_n-[\frac{1}{2}q_n]-1], \\
   I_2 &=& [ j  q_n-[\frac{1}{2}q_n], (j +1)q_n-[\frac{1}{2}q_n]-1 ].
\end{eqnarray*}
Let $\theta_m=\theta+m\alpha$ for $m\in I_1\cup I_2$. The set $\{\theta_m\}_{m\in I_1\cup I_2}$
consists of $2q_n$ elements.
\begin{lemma}\cite[Theorem B.5]{jl1}\label{lejuly2}
For any $\varepsilon>0$, $\{\theta_m\}_{m\in I_1\cup I_2}$ is $(1-t_1)\frac{\beta}{2}+ \varepsilon $ uniform.
\end{lemma}

\begin{lemma}\label{Le.r_j}
For $[\frac{\ell}{q_n}]  \leq   |j| \leq 2\ell $, the following holds
\begin{equation}\label{r_j}
   r_j\leq  \max\{ r_{j\pm1}\exp\{-(\ln \lambda-(1-t_1)\beta- \varepsilon)q_n\} \}.
\end{equation}
\end{lemma}
\begin{proof}

 By Lemma \ref{lejuly2} and   Lemma  \ref{Le.Uniform}, there exists some $j_0$ with  $j_0\in I_1\cup I_2$
    such that
      $ \theta_{j _0}\notin  A_{2q_n-1,\ln\lambda-(1-t_1)\frac{\beta}{2}- \varepsilon}$.

      Firstly, we assume $j_0\in I_2$.
      \par
      Set $I=[j_0-q_n+1,j_0+q_n-1]=[x_1,x_2]$.  
      Denote by $x_1'=x_1-1$ and $x_2'=x_2+1$.
      By  (\ref{Cramer1}), (\ref{Cramer2}) and  (\ref{Numerator}),
 it is easy to verify
\begin{equation*}
|G_I(jq_n+r,x_i)|\leq e^{(\ln\lambda+\varepsilon )(2q_n-2-|jq_n+r-x_i|)-(2q_n-1)(\ln\lambda -(1-t_1)\frac{\beta}{2}-\varepsilon)}.
\end{equation*}
Using (\ref{Block}), we obtain
\begin{equation}\label{Iterationr_j}
    |\phi(j q_n+r)|  \leq \sum_{i=1,2} e^{((1-t_1)\beta +\varepsilon)q_n}|\phi(x_i')|e^{-|jq_n+r-x_i|\ln \lambda }.
\end{equation}
Let $d_j^i=  |x_i-jq_n|  $, $i=1,2$. It is easy to check that
\begin{equation}\label{G.Distance}
   |jq_n+r-x_i|+d_j^i,|jq_n+r-x_i|+d_{j\pm1}^i \geq q_n-|r|,
\end{equation}
and
\begin{equation}\label{G.Distanceadd}
    |jq_n+r-x_i|+d_{j\pm2}^i \geq 2q_n-|r|.
\end{equation}
If $ {\rm dist}(x_i,q_n\mathbb{Z})\geq     10\varepsilon q_n $, then we replace $\phi(x'_i)$  in  (\ref{Iterationr_j}) with (\ref{Intervalk}).
If   ${\rm dist}(x_i,q_n\mathbb{Z})\leq 10\varepsilon q_n$, then we replace $\phi(x_i')$  in  (\ref{Iterationr_j}) with some proper $r_j$.
Combining with (\ref{G.Distance}) and (\ref{G.Distanceadd}),
we have
\begin{eqnarray*}
   r_j&\leq&\max\{r_{j\pm1}\exp\{-(\ln \lambda-(1-t_1)\beta- \varepsilon)q_n\}   ,r_{j}\exp\{-(\ln \lambda-(1-t_1)\beta- \varepsilon)q_n\}, \\
  && r_{j\pm 2}\exp\{-(2\ln \lambda-(1-t_1)\beta- \varepsilon)q_n\}\}.
\end{eqnarray*}
However $ r_j\leq r_{j}\exp\{-(\ln \lambda-(1-t_1)\beta- \varepsilon)q_n\}   $ can not happen, then we must have
\begin{equation}\label{G.Secondadd1}
   r_j\leq \max\{r_{j\pm1}\exp\{-(\ln \lambda-(1-t_1)\beta- \varepsilon)q_n\}   ,r_{j\pm 2}\exp\{-(2\ln \lambda-(1-t_1)\beta- \varepsilon)q_n\}\}.
\end{equation}

By \eqref{g500}, one has that
\begin{equation}\label{gnov191}
    r_{j+2}\leq e^{(\ln\lambda+\varepsilon) q_n}  r_{j+1}
\end{equation}
and
\begin{equation}\label{gnov192}
    r_{j-2}\leq e^{(\ln\lambda+\varepsilon) q_n}  r_{j-1}.
\end{equation}

By \eqref{G.Secondadd1}, \eqref{gnov191} and \eqref{gnov192}, one has that
\begin{equation}\label{gnov193}
   r_j\leq \max\{r_{j\pm1}\exp\{-(\ln \lambda-(1-t_1)\beta- \varepsilon)q_n\}  \}.
\end{equation}
If  $j_0\in I_1$,  then (\ref{gnov193}) holds for  $j=0$.  Combining with (\ref{gen}),   we  get
 $|\phi(0)|,|\phi(1)|\leq   \exp\{-(\ln \lambda-(1-t_1)\beta-\varepsilon)q_n\}$, this is contradicted to the fact  $\phi(0)^2+\phi(1)^2=1$.
This leads to  $j_0\in I_2$,
  then  Lemma \ref{Le.r_j} follows from (\ref{gnov193}).
\end{proof}

\begin{proof}[\bf Proof of Theorem \ref{Th.resonant}]

Without loss of generality, assume $k>0$.
Starting with $j=\ell$ in  (\ref{r_j}) and iterating  (\ref{r_j}) $\ell-[\frac{\ell}{q_n}]$ times,
we  have that
\begin{equation}\label{G.r_j}
 {r}_{\ell}\leq  (\ell+1) q_n\exp\{-(\ln \lambda-(1-t_1)\beta-\varepsilon) (1-\frac{1}{q_n})\ell q_n\}.
\end{equation}
This implies   Theorem \ref{Th.resonant}.
 \end{proof}
\section{Estimates on all solutions}
Let $\mu_{\delta_0}$ and $\mu_{\delta_1}$ be the spectral measures of, correspondingly, $\delta_0$ and $\delta_1$ for operator $H,$ and let $\mu=\mu_{\delta_0}+\mu_{\delta_1}$.

By basic spectral theory, we have that for any $\varphi\in\ell^2(\Z)$,  there are Borel functions $f_{\varphi}$ and $g_{\varphi}$
such that 
\begin{equation}\label{glast11}
 \mu_{\varphi}=f _{\varphi}\mu_{\delta_0}+g_{\varphi}\mu_{\delta_1}. 
\end{equation}
In particular,  
for any $\varphi\in\ell^2(\Z)$, $\mu_{\varphi}$ is absolutely continuous with respect to $\mu.$

 By   Schnol's theorem,  almost every $E$ with respect to the spectral measure $\mu$   is a generalized eigenvalue.
Furthermore, the following  stronger statement  holds \cite{ls99}.
\begin{theorem}\label{Th.General}
	For almost every $E$ with respect to the
	spectral measure $\mu$, there   exists  a solution  $u$ of equation $Hu=Eu$ such that for all $L\geq 2$,
	\begin{equation}\label{G.genral}
	||u||_{L,L} \leq {C}(E)L^{\frac{1}{2}}\ln L.
	\end{equation}
\end{theorem}

Let $S\subset \mathbb{R}$ be the full measure set in Theorem \ref{Th.General}, and assume  the solution satisfying \eqref{G.genral}  with $E\in S$ has initial conditions given by  $u(0,E)= \sin2\pi x_0(E)$ and $u(1,E)=-\cos2\pi x_0(E)$.  In other words, $ u_{x_0(E)}(\cdot,E)$ satisfies \eqref{G.genral}, where $u_x$ is defined in Section 2.2.

Since the spectrum of $H_{\lambda,\alpha,\theta}$ is in  $[-2-2\lambda,2+2\lambda]$,   we always assume $|E|\leq 2+2\lambda$.
\begin{lemma}\label{Le.4transfer}
For almost every $E$ with respect to $\mu$,  we have that for  large $L>0$,
  \begin{equation}\label{KKLKer1}
         ( \max_{x}||u_{x}(\cdot,E)||_{L,0})\cdot( \min_{x}||u_{x}(\cdot,E)||_{L,0}) \geq  L^{1+\frac{1}{2}\frac{\ln \lambda}{t_1\beta}- \varepsilon},
 \end{equation}
 and
 \begin{equation}\label{KKLKer2}
        (  \max_{x}||u_{x}(\cdot,E)||_{0,L})\cdot (\min_{x}||u_{x}(\cdot,E)||_{0,L} )\geq   L^{1+\frac{1}{2}\frac{\ln \lambda}{t_1\beta}- \varepsilon}.
 \end{equation}
\end{lemma}
\begin{proof}
 We only prove  \eqref{KKLKer1}, the proof of \eqref{KKLKer2} is similar.

Find $n$ such that  $4q_n\leq L<4q_{n+1}$.  Let $u_{x_1}(\cdot,E)$  be such that
  \begin{equation*}
    ||u_{x_1}(\cdot,E)||_{L,0}=\min_{x}||u_{x}(\cdot,E)||_{L,0}.
\end{equation*}

 Case 1: $   b_{n+1}\leq  20q_n^2q_{n+1}^{t_1}$ (recalling that $b_n=q_n^{t_2}$), then

 \begin{equation*}
 L\leq  4q_{n+1}\leq 4(20q^2_n)^{\frac{1}{t_2-t_1}}.
\end{equation*}
 Applying  (\ref{Intervalk}) with $\varphi=u_{x_1}$,  one has that
 \begin{equation*}
 |u_{x_1}([\frac{q_n}{2}])|,|u_{x_1}([\frac{q_n}{2}]-1)|\leq  \max\{\tilde{r}_0,\tilde{r}_{1}\}e^{-\frac{1}{2}(\ln\lambda-\varepsilon) q_n},
 \end{equation*}
  where $\tilde{r}_j= r_j^{u_{x_1}}=\sup_{|r|\leq 10\varepsilon }|u_{x_1}(jq_n+rq_n)|$, $j=0,1$.
 
 By the constancy of the Wronskian \eqref{gnov112}, one has  that
 \begin{equation*}
 \max\{   |u_{x_1+1/4}([\frac{q_n}{2}])|,|u_{x_1+1/4}([\frac{q_n}{2}]-1)|\}\geq \frac{e^{ \frac{1}{2} (\ln\lambda-\varepsilon) q_n}}{\hat{r}_{0}},
 \end{equation*}
 where  $\hat{r}_{0}=\max\{\tilde{r}_0,\tilde{r}_{1}\}$.
 This leads that
 \begin{equation}\label{Gct12}
 ||u_{x_1+1/4}||_{L,0} ^2\geq   \frac{e^{  (\ln\lambda-\varepsilon) q_n}}{\hat{r}_{0}^2}.
 \end{equation}
 
 Clearly by \eqref{g500}, one has that
 \begin{equation}\label{Gct13}
  \hat{r}_{0}^2 \leq  e^{\varepsilon q_n} ||u_{x_1}||_{L,0}^2.
 \end{equation}
 By \eqref{Gct12} and \eqref{Gct13}, one has
 \begin{eqnarray}
 \nonumber ||u_{x_1+1/4}||_{L,0} ^2||u_{x_1}||_{L,0}^2&\geq&   \ e^{  (\ln\lambda-\varepsilon) q_n}\label{glast7} \\
 &\geq&    L ^{2+\frac{\ln \lambda}{\beta t_1}-\varepsilon}\label{Gct22}.
 \end{eqnarray}

This implies
\eqref{KKLKer1}.

 Case 2:
 $   20q_n^2q_{n+1}^{t_1}\leq  b_{n+1}$, $L>20q_n^2q_{n+1}^{t_1}$. In this case, there exists some $$k\in[3q_n^2q_{n+1}^{t_1}-q_n,3q_n^2q_{n+1}^{t_1}]$$ that is $n$-resonant.  Recall that $u_{x_0}$ satisfies
 (\ref{G.genral}).    By Theorem \ref{Th.resonant}, it is easy to see
 \begin{eqnarray*}
  |u_{x_0}(k)|,|u_{x_0}(k-1)| &\leq &  \exp\{-(\ln \lambda-(1-t_1)\beta-\varepsilon)k\}  \\
    &\leq&  \exp\{-(\ln \lambda-(1-t_1)\beta-\varepsilon)2q_n^2q_{n+1}^{t_1}\} \\
    &\leq&  \exp\{-(\ln \lambda-(1-t_1)\beta-\varepsilon)\frac{2}{4^{t_1}}q_n^2L^{t_1}\},
 \end{eqnarray*}
 where the third inequality uses the fact that $L< 4q_{n+1}$.
 
 By the constancy of the Wronskian \eqref{gnov112}, we have
 \begin{equation*}
 \max{ \{|u_{x_0+1/4}(k)|,|u_{x_0+1/4}(k-1)|\}}\geq  \exp\{(\ln \lambda-(1-t_1)\beta-\varepsilon)\frac{2}{4^{t_1}}q_n^2L^{t_1}\}.
 \end{equation*}
 This clearly leads to
 \begin{equation*}
 ||u_{x_0+1/4}(\cdot,E)||_{L,0}\geq   L^{1+\frac{1}{2}\frac{\ln \lambda}{t_1\beta}-\varepsilon}.
 \end{equation*}
 This implies
 \begin{eqnarray*}
 	\max_{x}||u_{x}(\cdot,E)||_{L,0}\cdot \min_{x}||u_{x}(\cdot,E)||_{L,0}  &\geq&  \max_{x}||u_{x}(\cdot,E)||_{L,0} \\
 	&\geq&  ||u_{x_0+1/4}(\cdot,E)||_{L,0}\\
 	&\geq&   L^{1+\frac{1}{2}\frac{\ln \lambda}{t_1\beta}-\varepsilon}.
 \end{eqnarray*}

 \par
 Case 3:
 $   L\leq 20q_n^2q_{n+1}^{t_1}\leq  b_{n+1}$. We may assume $  q_n\leq L^{\varepsilon}$, otherwise this was covered by Case 1 (see \eqref{glast7}).
 Let $p= [ \frac{L}{4q_n}]>1$.
 Applying  Lemma \ref{Le.resonant1} with  $\varphi=u_{x_1}$, we have that for $j=0,\cdots p-1$,
 \begin{equation*}
    |u_{x_1}(jq_n+[\frac{q_n}{2}])|,|u_{x_1}(jq_n+[\frac{q_n}{2}]-1)|\leq  \max\{\tilde{r}_j,\tilde{r}_{j+1}\}e^{-\frac{1}{2}(\ln\lambda-\varepsilon) q_n},
 \end{equation*}
  where $\tilde{r}_j= r_j^{u_{x_1}}=\sup_{|r|\leq 10\varepsilon }|u_{x_1}(jq_n+rq_n)|$.

 By the constancy of the Wronskian again \eqref{gnov112}, one has that for  $j=0,1,\cdots,p-1$,
 \begin{equation*}
   \max\{   |u_{x_1+1/4}(jq_n+[\frac{q_n}{2}])|,|u_{x_1+1/4}(jq_n+[\frac{q_n}{2}]-1)|\}\geq \frac{e^{ \frac{1}{2} (\ln\lambda-\varepsilon) q_n}}{\hat{r}_{j}},
 \end{equation*}
 where  $\hat{r}_{j}=\max\{\tilde{r}_j,\tilde{r}_{j+1}\}$.
 This leads to
 \begin{equation}\label{Gct}
     ||u_{x_1+1/4}||_{L,0} ^2\geq \sum_{j=0}^{p-1}  \frac{e^{  (\ln\lambda-\varepsilon) q_n}}{\hat{r}_{j}^2}.
 \end{equation}
\par
Clearly by \eqref{g500}, one has
 \begin{equation}\label{Gct1}
    \sum_{j=0}^{p-1} \hat{r}_{j}^2 \leq  e^{\varepsilon q_n} ||u_{x_1}||_{L,0}^2.
 \end{equation}
 By \eqref{Gct} and \eqref{Gct1}, one has
 \begin{eqnarray}
 \nonumber ||u_{x_1+1/4}||_{L,0} ^2||u_{x_1}||_{L,0}^2&\geq&    e^{  (\ln\lambda-\varepsilon) q_n} (\sum_{j=0}^{p-1} \frac{1}{\hat{r}_{j}^2} )(\sum_{j=0}^{p-1} \hat{r}_{j}^2 )\\
  \nonumber &\geq&   \frac{e^{  (\ln\lambda-\varepsilon) q_n}}{q_n^2} L^2\\
    &\geq&    L ^{2+\frac{\ln \lambda}{\beta t_1}-\varepsilon},\label{Gct2}
 \end{eqnarray}
 where the third inequality holds by the fact that $L\leq 20q_n^2q_{n+1}^{t_1}$ and $  q_n^2\leq L^{\varepsilon}$.
 Now  (\ref{KKLKer1}) follows from \eqref{Gct2}.

\end{proof}
 
\section{  Proof of Theorems \ref{Maintheoremp}  and \ref{Maintheoremp1} }
Recall that  $\tilde {m}_j(z)$  $j=1,2$ are the $m$ functions of the right and left half line problems with boundary condition
\eqref{mbd}.

By   \eqref{glast12}, \eqref{glast13}, \eqref{glast14} and \eqref{glast15}, one has that
\begin{equation}\label{lastg6}
    M_{1}=\frac{\tilde {m}_1 \tilde {m}_2\sin^22\pi x_0-\tilde {m}_1\cos2\pi x_0\sin2\pi x_0+\tilde {m}_2\cos2\pi x_0\sin2\pi x_0-\cos^22\pi x_0}{\tilde {m}_1+\tilde {m}_2},
\end{equation}
and
\begin{equation*}
    M_{2}=\frac{\tilde {m}_1 \tilde {m}_2\cos^22\pi x_0+\tilde {m}_1\cos2\pi x_0\sin2\pi x_0-\tilde {m}_2\cos2\pi x_0\sin2\pi x_0-\sin^22\pi x_0}{\tilde {m}_1+\tilde {m}_2}.
\end{equation*}
 
\begin{lemma}\label{Le.im}
	Let $t\in(0,1)$. 
Assume that $\Im \tilde {m}_j(E+i\eps)\geq  \eps^{-t}$ for  small $\epsilon>0$ (the smallness depends on $E$ and $t$). Then
for    small $\epsilon>0$ , we have
\begin{equation}\label{G.scalM1}
    \Im M_1(E+i\eps)\geq  \eps^{-t},\;\;{\rm if \;\;} 2x_0\notin\Z,
\end{equation}
and 
\begin{equation}\label{G.scalM2}
    \Im M_2(E+i\eps)\geq   \eps^{-t} \;\;{\rm if \;\;} 2x_0\notin \frac{1}{2}+\Z.
\end{equation}
\end{lemma}
\begin{proof}

Let $\tilde {m}_j(z)=a_j(z)+ib_j(z)$, $j=1,2$, then
\begin{equation}\label{lastg5}
b_j(z)=b_j(E+i\epsilon)\geq  \epsilon^{-t},j=1,2.
\end{equation}

 We only give the proof of (\ref{G.scalM1}), the proof \eqref{G.scalM2} being similar.


Note that for small $\epsilon$, 
\begin{equation}\label{G.smalleps}
    b_1,b_2\geq 2\left|\frac{\cos2\pi x_0}{\sin2\pi x_0}\right|.
\end{equation}
Therefore,  we have
\begin{eqnarray*}
   \Im M_1 &=& \frac{(a_2^2b_1+a_1^2b_2+b_1^2b_2+b_1b_2^2)\sin^22\pi x_0-2\sin2\pi x_0\cos2\pi x_0(a_2b_1-a_1b_2)}{(a_1+a_2)^2+(b_1+b_2)^2} \\
   &&+\frac{(b_1+b_2)\cos^22\pi x_0}{(a_1+a_2)^2+(b_1+b_2)^2}\\
   &\geq& \frac{(a_2^2b_1+a_1^2b_2+b_1^2b_2+b_1b_2^2)\sin^22\pi x_0-2\sin2\pi x_0\cos2\pi x_0(a_2b_1-a_1b_2)}{(a_1+a_2)^2+(b_1+b_2)^2}\\
   &\geq& \frac{(a_2^2b_1+a_1^2b_2+b_1^2b_2+b_1b_2^2)\sin^22\pi x_0-|\sin2\pi x_0\cos2\pi x_0|(a_1^2+a_2^2+b_1^2+b_2^2)}{(a_1+a_2)^2+(b_1+b_2)^2}\\
     &\geq& \frac{\sin^22\pi x_0}{2}\frac{(a_2^2b_1+a_1^2b_2+b_1^2b_2+b_1b_2^2)}{(a_1+a_2)^2+(b_1+b_2)^2}\\
      &\geq&\frac{\sin^22\pi x_0}{16}\min\{b_1,b_2\},
\end{eqnarray*}
where the first inequality holds by the fact that $b_1,b_2\geq 0$, the second one holds by Cauchy–Schwarz inequality, the third one holds by
(\ref{G.smalleps}).
This leads to (\ref{G.scalM1}).
 
\end{proof}

\begin{lemma}\label{Le.im1}
Fix  any $0<t<\frac{\ln \lambda}{2\beta-\ln\lambda}$.  Then for almost every $E$ with respect to measure $\mu_{\delta_1}+\mu_{\delta_0}$,  one has that for small $\epsilon>0$,
\begin{equation*}
    \Im \tilde {m}_{j}(E+i\eps)\geq   \eps^{-t},j=1,2.
\end{equation*}
\begin{proof}
	We only give the proof for $j=1$. The proof for $j=2$ is similar.

Let 

\begin{eqnarray*}
  a(L) &=& ||u_{1/4+x_0}||^2_{L,0} ,\\
  b(L) &=& ||u_{x_0}||^2_{L,0}, \\
  \omega(L) &=& \max_{x}||u_x(\cdot,E)||_{L,0}\cdot\min_{x}||u_{x}(\cdot,E)||_{L,0}.
\end{eqnarray*}
Recall that $u_{x_0}$ satisfies \eqref{G.genral}.  Then, we have
\begin{equation}\label{G.bL}
    b(L)\leq C(E)L^{1+\varepsilon}.
\end{equation}
Let $L(\eps)$ be such that
\begin{equation}\label{G.L}
  \omega(L(\eps))=\frac{1}{\eps}.
\end{equation}
By Lemma \ref{Le.4transfer}, one has that for   small $\epsilon$, 
\begin{equation*}
    \frac{1}{\eps}=
          \max_{x}||u_{x}(\cdot,E)||_{L(\eps),0}\cdot \min_{x}||u_{x}(\cdot,E)||_{L(\eps),0} \geq  L(\epsilon)^{1+g},
\end{equation*}
where $g=\frac{1}{2}\frac{\ln \lambda}{t_1\beta}- \varepsilon$. Thus
\begin{equation}\label{G.UL}
    L(\eps)\leq \eps^{-\frac{1}{1+g}}.
\end{equation}
By  Theorem \ref{kklthm},  we have that
\begin{equation}\label{gnov111}
    \Im \tilde {m}_1(E+i\eps)\geq \frac{1}{C} \frac{1}{\eps} \frac{1}{b(L(\eps))}.
\end{equation}
By (\ref{G.bL}) (\ref{G.L}), (\ref{G.UL}) and \eqref{gnov111}, we have that
\begin{eqnarray*}
 \Im \tilde {m}_1(E+i\epsilon )&\geq& \frac{1}{C} \frac{1}{ \epsilon L(\eps)^{1+\varepsilon}}   \\
   &\geq&  \eps^{-t_0} ,
\end{eqnarray*}
where $t_0=1-(1+\varepsilon)/(1+g)$.
 
Now Lemma  \ref{Le.im1} holds by letting $\varepsilon,\sigma$ go to $0$.
\end{proof}

\begin{lemma}\label{lelast}
	Recall that  $u_{x_0(E)}$ satisfies  \eqref{G.genral}.
	Let 
	\begin{equation}
	S_0=\{E: 2x_0(E) \in\Z \}
	\end{equation}
	and 
	\begin{equation}
	S_1=\{E: 2x_0(E) \in\frac{1}{2}+\Z \}.
	\end{equation}
	Then we have
	\begin{equation}\label{lastg}
	\mu_{\delta_0}(S_0)=0
	\end{equation}
	and
	\begin{equation}\label{lastg1}
	\mu_{\delta_1}(S_1)=0.
	\end{equation}
\end{lemma}
\begin{proof}
	We only prove \eqref{lastg}. The proof of \eqref{lastg1} is similar.
	By Lemma \ref{Le.im1}, we have that for almost every $E$ respect to $\mu_{\delta_0}+\mu_{\delta_1}$ (recall that $\tilde {m}_j(z)=a_j(z)+ib_j(z)$, $j=1,2$),
	\begin{equation}
b_1(z)	=b_1(E+i\epsilon)\geq \epsilon^{-t}, 	b_2(z)=b_2(E+i\epsilon)\geq \epsilon^{-t}.
	\end{equation}
	
	Assume that $2x_0(E)\in\Z$. Then  
	by \eqref{lastg6},  one has that
	\begin{equation*}
	\Im M_1(z)=\frac{b_1+b_2}{(a_1+a_2)^2+(b_1+b_2)^2},
	\end{equation*}
	which implies that 
		\begin{equation*}
\lim_{\epsilon\to 0+}	\Im M_1(z)=0.
	\end{equation*}
	By basic spectral theory, $\mu_{\delta_0}(S_1)=0$.
\end{proof}

\end{lemma}
 Lemmas    \ref{Le.im}, \ref{Le.im1}  and \ref{lelast} imply
\begin{theorem}\label{Th.scaling}
Fix  any $0<t<\frac{\ln \lambda}{2\beta-\ln\lambda}$. Then  for almost every $E$ with respect to spectral measure $\mu_{\delta_0}+\mu_{\delta_1}$, the following holds for small $\epsilon>0$,
\begin{equation}\label{G.imM1}
    \Im M_1(E+i\epsilon)\geq  \epsilon^{-t}, 
\end{equation}
and 
\begin{equation}\label{G.imM2}
    \Im M_2(E+i\eps)\geq  \eps^{-t}.
\end{equation}
\end{theorem}

\begin{proof}[\bf Proof of Theorem \ref{Maintheoremp}]
	 
By Theorem \ref{Th.scaling}, we have for any $\varsigma \in(\frac{2\beta-2\ln\lambda}{2\beta-\ln\lambda},1)$, and almost every $E$  with respect to spectral measure $\mu_{\delta_0}+\mu_{\delta_1}$, 
\begin{equation}\label{mera04E}
\eps^2 \int_{\R} {d \mu_{\delta_j}(y) \over |E-y|^2 +\eps^2}= \eps\Im M_j(E+i\eps)\ge  \eps^\varsigma,
\end{equation}
where $j=0,1$.
Applying Corollary \ref{CD+} with $m=2$, one has that
\begin{equation*}
   {\rm dim}^+_P(\mu_{\delta_j}) \le m \varsigma /(m-\varsigma), j=0,1.
\end{equation*}
It implies for $j=0,1$,
\begin{eqnarray}
   {\rm dim}^+_P(\mu_{\delta_j}) &\leq&  \frac{2\frac{2\beta-2\ln\lambda}{2\beta-\ln\lambda}}{2-\frac{2\beta-2\ln\lambda}{2\beta-\ln\lambda}}\nonumber\\
   &=& \frac{2\beta-2\ln\lambda}{\beta}.\label{glast10}
\end{eqnarray}
By \eqref{glast11} and \eqref{glast10}, one has that  for any $\varphi\in \ell^2(\Z)$,
\begin{equation*}
   {\rm dim}^+_P(\mu_{\varphi})\leq \frac{2\beta-2\ln\lambda}{\beta}.
\end{equation*}
This completes the proof of Theorem \ref{Maintheoremp}.
\end{proof}

 \begin{proof}[\bf Proof of Theorem \ref{Maintheoremp1}]
 	It follows from Theorem \ref{upD+},  \eqref{mera04E} and \eqref{glast11}.
 \end{proof}
 
 \section*{Acknowledgments}
W. Liu was a 2024-2025 Simons fellow.
W. Liu would like to thank the Department of Mathematics at UC Berkeley for their hospitality, where this work was completed during his visit.

WL's work was supported in part by NSF DMS-2246031,  DMS-2052572 and  DMS-2000345. 
 SJ’s
work was supported by NSF DMS-2052899, DMS-2155211, and Simons 896624.  SJ and WL are also grateful to UC Irvine and the Isaac Newton Institute for Mathematical Sciences,
Cambridge, for its hospitality, supported by EPSRC Grant Number EP/K032208/1, during
the 2015 programme Periodic and Ergodic Spectral Problems, where the bulk of this work was done.

\end{document}